\documentclass[11pt]{amsart}
\usepackage{amssymb}
\usepackage{graphicx}
\usepackage{xcolor} 
\usepackage{fullpage} 
\usepackage{amsmath}
\usepackage{amsthm}
\usepackage{hyperref}
\usepackage{verbatim}
\usepackage{enumitem}

\newtheorem{theorem}{Theorem}
\newtheorem{lemma}[theorem]{Lemma}
\newtheorem{prop}[theorem]{Proposition}

\newtheorem*{theorem*}{Theorem}
\newtheorem*{lemma*}{Lemma}

\theoremstyle{definition}

\newtheorem{remark}{Remark}

\newtheorem*{remark*}{Remark}
\newtheorem*{example*}{Example}
\newtheorem*{er*}{Examples and Remarks}

\newcommand{\nrm}[1]{\Vert#1\Vert}

\newcommand{\nnrm}[1]{{\vert\kern-0.25ex\vert\kern-0.25ex\vert #1 \vert\kern-0.25ex\vert\kern-0.25ex\vert}}

\newcommand{\supp}{{\mathrm{supp}}\,}

\newcommand{\lap}{\Delta}

\newcommand{\rd}{\partial}
\newcommand{\nb}{\nabla}

\newcommand{\dlt}{\delta}

\newcommand{\eps}{\epsilon}

\newcommand{\omg}{\omega}

\begin{document}

\title[vortex thinning]
{Enstrophy dissipation and  vortex thinning 
for the incompressible 2D Navier-Stokes equations} 
 
\author{In-Jee Jeong}
\address{School of Mathematics, Korea Institute for Advanced Study, Seoul, Republic of Korea}
\email{ijeong@kias.re.kr}

\author{Tsuyoshi Yoneda} 
\address{Graduate School of Mathematical Sciences, University of Tokyo, Komaba 3-8-1 Meguro, 
Tokyo 153-8914, Japan } 
\email{yoneda@ms.u-tokyo.ac.jp} 


\date{\today} 


\begin{abstract} {Vortex thinning is one of the main mechanisms of two-dimensional turbulence.}
By direct numerical simulation to the two-dimensional Navier-Stokes equations with small-scale forcing and large-scale damping, Xiao-Wan-Chen-Eyink (2009) found an evidence that inverse energy cascade may proceed 
with the vortex thinning mechanism. 
On the other hand, Alexakis-Doering (2006) calculated upper bound of the bulk averaged  enstrophy dissipation rate of the steady-state two dimensional turbulence.
{In this paper, we show that vortex thinning induces enhanced dissipation with strictly slower vanishing order of the enstrophy dissipation than $Re^{-1}$.}
\end{abstract}

\maketitle

\section{Introduction} \label{sec:Intro} 

The mechanism of forward and inverse energy cascade of two-dimensional turbulence has been attracting many physicists but nevertheless is not well clarified so far.
Theoretical studies of 2D turbulence usually employ statistics and sometimes
impose assumptions on homogeneity isotropy.
Alexakis and  Doering \cite{AD} derived some rigorous upper
bounds for the long time averaged bulk energy and enstrophy dissipation rates for 2D
statistically stationary flows sustained by a variety of driving forces.
In particular, they showed the enstrophy dissipation  vanishes in the order $Re^{-1}$
 when the external force is only on the single scale.
They are  emphasizing  that, 
 in this $Re^{-1}$ scaling, the flow exhibits laminar behavior, 
since
 energy is concentrated at relatively long length 
scales (independent of Reynolds number).
Thus, their result tells us that, at least, 2D turbulence behavior must provide vanishing order which is strictly slower than the order $Re^{-1}$.
Let us explain more precisely.
The two dimensional Navier-Stokes equations are described as follows:
\begin{equation} \label{NS}
\left\{
\begin{aligned} 
&\partial_tu -\nu\Delta u+ u\cdot\nabla u = - \nabla p+f, 
\qquad 
t \geq 0, \; x \in 
(\mathbb{R}/2\mathbb{Z})^2
\\ 
&\mathrm{div}\, u = 0 
\\ 
&u(t=0) = u_0 
\end{aligned}
\right.
\end{equation} 
where $u=u(t,x)=(u_1(t,x),u_2(t,x))$, $p=p(t,x)$ and  $f=(f_1(t,x),f_2(t,x))$ denote the velocity field, the pressure function of the fluid and the external force respectively. The case $\nu=0$ is called the Euler equations. {Introducing the vorticity $\omg=\nb\times u$, we can rewrite \eqref{NS} as the following 2D vorticity equations: 
	\begin{equation}\label{NS:vort}
	\begin{split}
	\partial_t\omega -\nu\Delta \omega+ (u\cdot\nabla) \omega = \nb\times f,\quad x \in \mathbb{T}^2 := (\mathbb{R}/2\mathbb{Z})^2, 
	\end{split}
	\end{equation} where the velocity $u$ is determined by the (periodic) 2D Biot-Savart law: \begin{equation*}
	\begin{split}
	u(t,x) = \int_{\mathbb{T}^2} K_2 \large( x - y \large) \omega(t,y) \, dy, 
	\end{split}
	\end{equation*} with \begin{equation*}
	\begin{split}
	K_2(x)  = \frac{1}{2\pi} \frac{(-x_2,x_1)}{|x|^2} \quad (\text{with reflections}).
	\end{split}
	\end{equation*}}

 Assume that the external force $f$ involves only a single length scale,
namely, $-\Delta f\sim k_f^2 f$.
Assume further that the flow is a statistical steady state of body forced
two-dimensional turbulence.   
Then the solution $u$ (depending on $\nu$)  to the Navier-Stokes equations \eqref{NS}
satisfy 
\begin{equation*}
\nu\langle|\nabla \omega|^2\rangle=\chi\leq \nu k^4_f U^2, 
\end{equation*}
where $\chi$ is the enstrophy dissipation, $\langle\cdot\rangle$ is some averaging (in this paper we regard it as both space and time averages) and $U$ is the representative velocity.
Note that, mathematically, if the enstrophy is finite, then this $\chi$ is always vanishing for $\nu\to 0$ (see \cite[Proposition 2]{FML} for example).

The aim of this paper is to find a specific vorticity  {configuration}
which provides strictly slower vanishing order of the enstrophy dissipation than $Re^{-1}$, and in this case ``vortex thinning" should be a strong candidate for it. 
 {To be more precise,} Xiao-Wan-Chen-Eyink (\cite{XWCE}) investigated inverse energy cascade in steady-state two-dimensional turbulence by direct numerical simulation of the two-dimensional Navier-Stokes equations with small scale forcing and  large scale damping. 
In their numerical work, they used an alternative  equation  (with a damping term),
and found strong evidence that inverse energy cascade may proceed with vortex thinning mechanism.
According to their evidence, there is a tensile turbulent stress in directions parallel to the isolines of small-scale vorticity.  
Thus the small-scale circular vortex will be stretched into elliptical shape, which is nothing more than ``vortex thinning" phenomena.

In this paper we focus on a sequence of prescribed initial vorticity which
provides vortex thinning behavior, and show that 
 the vortex thinning behavior
provides strictly slower vanishing order  {of the enstrophy dissipation} than $Re^{-1}$.
Now we formulate mathematically this enstrophy dissipation problem more precisely. 
Let us find a sequence of viscosity $\{\nu_n\}_n$ ($\nu_n\to 0$), a sequence of initial vorticity $\{\omega_{0,n}^{\nu_n}\}_n\subset H^1(\mathbb{T}^2)\cap C^\infty(\mathbb{T}^2)$ with $\|\omega_{0,n}^{\nu_n}\|_{H^1}\lesssim 1$
and a function $\varphi$ with $\varphi(\nu)\to 0$ ($\nu\to 0$),
 such that the corresponding solutions to the Navier-Stokes equations $\{\omega^{\nu_n}_n\}_n$ satisfy the following: for any $T>0$ fixed,
\begin{equation}\label{math-enstrophy-dissipation}
\liminf_{n\to \infty}\, \varphi(\nu_n) \frac{1}{T}\int_0^{T}\int_{\mathbb T^2}|\nabla\omega^{\nu_n}_n(t,x)|^2dxdt >0.
\end{equation}  {Since we are interested in a finite-time result, we shall neglect the external force and take $Re=\nu^{-1}$.}

Note that, in physics, $H^1$-norm of vorticity is called ``palinstrophy" (see \cite[Section 8]{L} for example).
The above formulation is the most strict one, since we are imposing 
 uniformly boundedness $\|\omega_{0,n}^{\nu_n}\|_{H^1}\lesssim 1$ and taking arbitrary time $T>0$. These are necessary to avoid ``trivial enstrophy dissipation": If we choose $\{\omega_{0,n}\}_n$ satisfying $\|\omega_{0,n}\|_{H^1}\to\infty$,
and choose $T_n$ $(T_n\to 0, n\to\infty$) to be 
$\sup_{0<t<T_n}\nu_n\|\omega_{0,n}-\omega_{0,n}^{\nu_n}(t)\|^2_{H^1}<\epsilon$ (for sufficiently small $\epsilon>0$) with  $\nu_n\sim\|\omega_{0,n}\|_{H^1}^{-2}$,
then we trivially have (c.f. \cite[Section 6]{FML})
\begin{equation}
\lim_{n\to \infty}\, \nu_n \frac{1}{T_n}\int_0^{T_n}\int_{\mathbb{T}^2}|\nabla\omega^{\nu_n}_n(t,x)|^2dxdt \sim 1.
\end{equation}
This is nothing more than trivial enstrophy dissipation we want to avoid.

We now construct our initial data.  Given any smooth radial bump function $0 \leq \phi \leq 1$ with support in the ball $B(0,1/4)$ 
let 
\begin{align*} 
\phi_0(x_1, x_2) 
= 
\sum_{\varepsilon_1, \varepsilon_2 = \pm 1} 
\varepsilon_1 \varepsilon_2 \phi(x_1 {-} \varepsilon_1, x_2 {-} \varepsilon_2). 
\end{align*} 
We assume further that $\phi = 1$ in the ball $B(0,1/8)$. Clearly, the function $\phi_0$ is odd with respect to both $x_1$ and $x_2$.  {Moreover, given $\ell_0\ge 10$ (to be determined later independently of $n$; see the proof of Proposition  \ref{prop:LLD-bubbles} below), we take $\tilde{\omg}_0$ to be a smooth function which is odd-odd and satisfy $\tilde{\omg}_0 = 0$ on $[0,1]^2\backslash [2^{-\ell_0+2}, 1- 2^{-\ell_0+2}]^2$ and $\tilde{\omg}_0=1$ on $[2^{-\ell_0+3}, 1- 2^{-\ell_0+3}]^2$.} Given some non-negative and decreasing sequence of reals $\{ a_\ell \}_{\ell\ge1}$, we define 
\begin{align}\label{eq:bubbles} 
\omega_{0,n}(x)  
=  
\phi_{2n}(x) + {\tilde{\omg}_0(x)} + \sum_{\ell=\ell_0}^{n} a_\ell \phi_\ell(x),
\end{align} 
%
where 
$\phi_\ell(x) =  \phi_0 (2^\ell x)$.
Note that the supports of $\phi_\ell$ are disjoint.   
{
We remark that $\tilde\omega_0$ is necessary to guarantee \eqref{C}. The term $\sum_\ell a_\ell\phi_\ell$ (``Bourgain-Li'' bubbles after the work \cite{BL}) creates a large rate-of-strain tensor at the origin, which thins the smallest scale vortex blob $\phi_{2n}$.}
\begin{remark}
In the proof of our main result, we take \begin{equation*}
\begin{split}
a_\ell = \ell^{-1/2-\eps}\quad\text{for some small}\quad 0<\eps<1/4.
\end{split}
\end{equation*} With this choice, we see that $\|\omega_{0,n}\|_{H^{1}}  \lesssim 1$ for all $n$. 
\end{remark}
We  mention that  the relation between the vortex thinning process and palinstrophy ($H^1$-norm of the vorticity) has already been studied. 
In 
 \cite[Section 6.3]{AP} (see also \cite{MSMOM}), Ayala-Protas found a initial vorticity of 2D Navier-Stokes equations (with very high Reynolds number)  which attains maximum growth of palinstrophy, by using their own optimizing method.
They figured out that the initial vorticity has odd (in both $x_1$ and $x_2$) type of  symmetry 
with two different scales formation (which seems very similar to our initial vorticity setting). 
On the other hand, 
some of mathematicians have showed that 
there is an initial vorticity in $H^1$ such that the value of  
$\|\nabla \omega(t)\|_{L^2}$ (palinstrophy) to the 2D-Euler flow instantaneously blows up.
More precisely, 
 Bourgain-Li \cite{BL} and  Elgindi-Jeong \cite{EJ} constructed  solutions
to the 2D-Euler equations 
which exhibit norm inflation  in  $H^1$ 
(see also 
\cite{MY}). {Now we state our main result.}
\begin{theorem}[Lower bound on the enstrophy dissipation] \label{thm:main2}
Let $\{\omega^\nu_{n}(t)\}_{\nu,n}$ be the unique solution  of the 2D Navier-Stokes equations with viscosity $\nu>0$ and initial data  $\omega_{0,n} $ given in \eqref{eq:bubbles}. Then, there exist a continuous function $\varphi(\nu)$ with $\varphi(\nu)\to 0$ as $\nu\to 0$ and a sequence of viscosity constants $\nu_{n} > 0$ converging to zero, 
such that for any fixed $T>0$,
\begin{equation}\label{zeroth-law-estimate}
\liminf_{n}\, \varphi(\nu_{n})  \frac{1}{T}\int_0^T\int_{\mathbb T^2}|\nabla \omega^{\nu_{n}}_{n}(t,x)|^2dxdt\gtrsim 1.
\end{equation}
\end{theorem}

\begin{remark} In the proof,  one sees that the growth of $\nrm{\nabla \omega^{\nu_{n}}_{n}}_{L^2}$ comes from vortex thinning of the smallest scale bubble $\phi_{2n}$ in \eqref{eq:bubbles}. The function $\varphi$ can be chosen as $\varphi(\nu) = (\log \frac{1}{\nu})^{-c_0}$ for some absolute constant $c_0>0$. Finally, note that the sequence of initial data in \eqref{eq:bubbles} is convergent weakly in $H^1$, 
and moreover,
we can obtain the same conclusion for the following initial data:
\begin{equation*}
\omega_{0,n}(x)=\frac{\phi_{2n}(x)}{(S_n\delta)^{\frac{c_0}{4}}}+  {\tilde{\omg}_0(x)} {+}  {\sum_{\ell=\ell_0}^na_\ell\phi_\ell(x),}
\end{equation*}
where $S_n$, $\delta$ and $c_0$ are determined in Proposition \ref{prop:LLD-bubbles}.
In this case the smallest-scale vorticity $\frac{\phi_{2n}(x)}{(S_n\delta)^{\frac{c_0}{4}}}$
vanishes for $n\to\infty$.
\end{remark}

\begin{remark}
Using somewhat similar construction, the authors considered enhanced energy dissipation under the $2+\frac{1}{2}$-dimensional Navier-Stokes flow: small-scale horizontal vortex blob being stretched by large-scale, anti-parallel pairs of vertical vortex tubes \cite{JY}. 
The zeroth law  states that, in the limit of vanishing viscosity, the rate of kinetic energy dissipation for solutions to the 3D incompressible Navier-Stokes equations becomes nonzero.
The authors  proved a version of the zeroth law satisfying the above, which implies in particular enhanced dissipation. 
\end{remark}
 
\subsubsection*{Notation} We write $A \lesssim B$ if there is an absolute constant $c>0$ such that $A\le cB$. Similarly, $A\approx B$ if $A\lesssim B$ and $A\gtrsim B$.

\section{Proof of the main result}

Our result is a consequence of the ``large Lagrangian deformation'' for the 2D Euler equations, which was first established in a work of Bourgain and Li \cite{BL}. However, their idea works only for a short time and relies on a contradiction argument. Even worse, they were not able to obtain any quantitative growth rate. To overcome these difficulties, we employ an induction argument accompanied by a careful scaling control. For additional technical improvements achieved in this paper, see the remarks following the proof of Proposition \ref{prop:LLD-bubbles}. 

Let $\eta_{n}=(\eta_{n,1},\eta_{n,2})$ be the associated Lagrangian flow from $\omega_{0,n}$, which is given by \begin{equation*}
\begin{split}
\partial_t \eta_n(t,x)=u_{n}(t,\eta_n(t,x))\quad\text{with}\quad \eta_n(0,x)=x\in\mathbb{T}^2,
\end{split}
\end{equation*} 
where $u_n$ is the corresponding velocity  associated with $\omega_{0,n}$. The proof of our main result is based on the following two propositions.  {The first proposition makes precise what we mean by large Lagrangian deformation.}
\begin{prop}[Creation of large Lagrangian deformation]\label{prop:LLD-bubbles}
	Let $\{a_k\}_{k=1}^\infty$ be a bounded sequence of non-negative reals, and  $\omega_n(t) $ be the solution  {to the 2D Euler equation} with initial data as in \eqref{eq:bubbles} with associated Lagrangian flow $\eta_n$. Set $S_k := S_{k-1} + a_{k}$ for $k \ge 1$ with $S_0 := 1$. Assume that $S_k$ is divergent with $k$. Then, for some absolute constant $c_0 > 0$,  we have  \begin{equation}\label{eq:velgrad-large}
	\begin{split}
	-\frac{\rd u_{n,1}}{\rd x_1}(t,0)=\frac{\rd u_{n,2}}{\rd x_2}(t,0) \gtrsim \frac{1}{t} \mathbf{1}_{\{t\gtrsim S_n^{-1}\}}
	\end{split}
	\end{equation} and  \begin{equation}\label{eq:LLD-bubbles}
	\begin{split}
	\left|\frac{\partial\eta_{n,1}}{\partial x_1}(t,0)\right|=\left|\frac{\partial\eta_{n,2}}{\partial x_2}(t,0)\right|  \gtrsim (S_nt)^{c_0}
	\end{split}
	\end{equation}  for all $t \in [0,\dlt)$ with some absolute constant $\dlt>0$ and $n$ sufficiently large. 
\end{prop} 

\begin{prop}\label{prop:LLD2}
	Fix $a_{\ell} = \ell^{-\frac{1}{2}-\frac{1}{8}}$ and $S_n:= 1+ \sum_{\ell=1}^n a_\ell$. Let $\omega_{n}(t)$ be the solution of the 2D Euler equation with initial data $\omega_{0,n}$ defined in \eqref{eq:bubbles}. Then we have 
	\begin{equation}\label{small-scale-lower-bound}
	\nrm{\nabla \omega_{n}(t)}_{L^2} \gtrsim (S_nt)^{\frac{c_0}{2}}
	\end{equation} for $t \in [0,\dlt]$ with the same $\dlt>0$ as in Proposition \ref{prop:LLD-bubbles}.
\end{prop}

\begin{remark}
	Although the above propositions are stated in possibly small time interval $[0,\dlt]$, we can take $\dlt>0$ arbitrarily large simply by rescaling the initial data.
More precisely, we just choose $\omega_{0,n}$ to be sufficiently small
in the $L^\infty$-norm (nevertheless $S_n$ is always divergent).
 In the following, we shall assume that the propositions are valid for $[0,T]$ with some $T>0$ given.
\end{remark}

Given the above propositions, we complete the proof of the main theorem. 
\begin{proof}[Proof of Theorem \ref{thm:main2}]
Since initial data is always smooth,
we note that for each $\nu > 0$, the unique solution $\{\omega^\nu_{n }\}$ 
exists globally-in-time. From now on the multiplicative constants in the estimates will depend implicitly on $T>0$.
 From the maximum principle $\nrm{\omega_{n}(t)}_{L^\infty} = 1$, we have 
\begin{equation*}
\begin{split}
\nrm{\nabla u_n(t)}_{L^\infty} \lesssim \log \nrm{\omega_{0,n}}_{C^1} e^{Ct} \lesssim n 
\end{split}
\end{equation*} on the time interval $[0,T]$. 
Note that $\|\omega_{0,k}^S\|_{C^1}$ can be always controlled by taking sufficiently large $k$ (depending on $n$).
Then we have, from the classical estimate
(see \cite{BKM} for example) 
\begin{equation*}
\begin{split}
\nrm{\omega_n(t)}_{H^s} \lesssim \nrm{\omega_{0,n}(t)}_{H^s} e^{C(s)\int_0^t \nrm{\nabla u_n(\tau)}_{L^\infty} d\tau} 
\end{split}
\end{equation*} that \begin{equation*}
\begin{split}
\nrm{\omega_n(t)}_{H^s} \lesssim 2^{c(s)n}  
\end{split}
\end{equation*} for some constant $c(s) > 0$ depending only on $s > 1$ for any $s$ and $t \in [0,T]$. 
We note that the Navier-Stokes solutions satisfy the same bounds: \begin{equation*}
\begin{split}
\nrm{\omega_n^{\nu}(t)}_{H^s} \lesssim 2^{c(s)n}  
\end{split}
\end{equation*} with constant independent of $\nu > 0$. This is because we still have the maximum principle $\nrm{\omega^{\nu}_{n}(t)}_{L^\infty} \le 1$ for all $t \ge 0$ and the $H^s$ estimate holds \textit{a fortiori} for the Navier-Stokes. Taking $s>3$ and by the Sobolev embedding, we obtain
\begin{equation}\label{Sobolev-embedding}
\nrm{\nabla^2\omega_n(t)}_{L^\infty}+\nrm{\nabla^2\omega_n^{\nu}(t)}_{L^\infty} \lesssim 2^{cn}
\quad\text{for}\quad t\in[0,T]. 
\end{equation} 
This is elementary but is the key in the estimates below.
 
\subsection{$L^2$ inviscid limit estimate on the  velocity}

We compare the 2D Euler and Navier-Stokes equations  of the velocity:
  \begin{equation*}
\begin{split}
&\rd_t u_n^{\nu} + u_n^{\nu} \cdot \nabla u_n^{\nu} + \nabla p_n^{\nu} = \nu \lap u_n^{\nu} , \\
&\rd_t u_n + u_n\cdot\nabla u_n + \nabla p_n = 0. 
\end{split}
\end{equation*} 
By
\begin{equation*}
 \int (u^\nu\cdot\nabla)( u_n^\nu - u_n) \cdot (u_n^{\nu} - u_n)=0,  
\end{equation*}
then we see that 
 \begin{equation*}
\begin{split}
\frac{1}{2}\frac{d}{dt} \nrm{u_n^{\nu} - u_n}_{L^2}^2 + \int( u_n - u_n^{\nu}) \cdot \nabla u_n \cdot (u_n^{\nu} - u_n)  
= \nu \int \lap u_n^{\nu} \cdot (u_n^{\nu} - u_n).
\end{split}
\end{equation*} We handle the right hand side as follows: \begin{equation*}
\begin{split}
-\nu\int |\nabla u_n^{\nu}|^2 + \nu \int \nabla u_n^{\nu} : \nabla u_n.
\end{split}
\end{equation*} Then, using the  Cauchy-Schwarz inequality and Young's inequality, we have \begin{equation*}
\begin{split}
\frac{d}{dt} \nrm{u_n^{\nu} - u_n}_{L^2}^2  \lesssim \nrm{\nabla u_n}_{L^\infty} \nrm{u_n^{\nu} - u_n}_{L^2}^2 + \nu \nrm{\nabla u_n}_{L^2}^2 \lesssim n\nrm{u_n^{\nu} - u_n}_{L^2}^2 + \nu ,
\end{split}
\end{equation*} which gives \begin{equation*}
\begin{split}
\nrm{u_n^{\nu} - u_n}_{L^2}^2 \lesssim  \nu 2^{cn}
\end{split}
\end{equation*} for $t \in [0,T]$. 
 
\subsection{$H^1$ inviscid limit estimate }

This time, we consider the 2D Navier-Stokes solutions in the vorticity form and compare it with the corresponding Euler solutions:
\begin{equation*}
\begin{split}
&\rd_t\omega_n^{\nu} + u_n^{\nu}\cdot\nabla \omega_n^{\nu} = \nu\lap \omega_n^{\nu},  \\
&\rd_t\omega_n + u_n\cdot\nabla \omega_n = 0 . 
\end{split}
\end{equation*} Then using previous bounds, \begin{equation*}
\begin{split}
\frac{d}{dt} \nrm{\omega_n - \omega_n^{\nu}}_{L^2}^2 &\lesssim \nrm{\nabla \omega_n}_{L^\infty} \nrm{u_n - u_n^{\nu}}_{L^2} \nrm{\omega_n - \omega_n^{\nu}}_{L^2}   + \nu \nrm{\nabla\omega_n}_{L^2}^2 \\
&\lesssim\nu 2^{cn} \nrm{\omega_n - \omega_n^{\nu}}_{L^2} +\nu 2^{cn}.
\end{split}
\end{equation*}
Thus we have 
\begin{equation*}
\begin{split}
\nrm{\nabla u_n -\nabla u_n^{\nu}}_{L^2} \lesssim \nrm{\omega_n - \omega_n^{\nu}}_{L^2} \lesssim \nu2^{cn}.
\end{split}
\end{equation*} 

\subsection{$H^2$ inviscid limit estimate  }
Let $\partial$ be a spatial derivative.
Here, we consider the 2D Navier-Stokes solutions in the vorticity-gradient form and compare it with the corresponding Euler solutions:
\begin{equation*}
\begin{split}
&\rd_t\partial\omega_n^{\nu} + (\partial u_n^{\nu}\cdot\nabla) \omega_n^{\nu}+  (u_n^{\nu}\cdot\nabla)\partial \omega_n^{\nu} = \nu\lap \partial\omega_n^{\nu},  \\
&\rd_t\partial\omega_n + (\partial u_n\cdot\nabla) \omega_n+ (u_n\cdot\nabla)\partial \omega_n = 0 . 
\end{split}
\end{equation*} Again using previous bounds, \begin{equation*}
\begin{split}
\frac{d}{dt} \nrm{\partial(\omega_n - \omega_n^{\nu})}_{L^2}^2 &\lesssim \nrm{\nabla\partial\omega_n}_{L^\infty} \nrm{u_n- u_n^{\nu}}_{L^2} \nrm{\partial(\omega_n- \omega_n^{\nu})}_{L^2} \\
&
+\|\nabla\omega_n\|_{L^\infty}\|\partial(u_n-u_n^\nu)\|_{L^2}\|\omega_n^\nu-\omega_n\|_{L^2}
  + \nu \nrm{\omega_n}_{\dot H^{3/2}}^2 \\
&\lesssim\nu 2^{cn} \nrm{\nabla(\omega_n - \omega_n^{\nu})}_{L^2} +\nu 2^{cn}.
\end{split}
\end{equation*}
Here we choose  $\nu\sim 2^{-cn}$, this guarantees that \begin{equation*}
\begin{split}
 \nrm{\nabla(\omega_n - \omega_n^{\nu})}_{L^2} \lesssim 1.
\end{split}
\end{equation*}

\subsection{Completion of the proof}
Recall \eqref{small-scale-lower-bound}  {from Proposition \ref{prop:LLD-bubbles}}:\begin{equation*}
\begin{split}
\nrm{\nabla\omega_n(t)}_{L^2} \gtrsim (S_nt)^{c_0}
\end{split}
\end{equation*} for $t \in [0,T]$.
Then
\begin{equation*}
\begin{split}
\frac{1}{T}\int_0^T\|\nabla\omega_n(t)\|_{L^2}^2dt  \gtrsim_T S_n^{2c_0}.
\end{split}
\end{equation*} 
Then we finally have 
\begin{equation*}
\begin{split}
&\ \varphi(\nu_n)\frac{1}{T}\int_0^T\|\nabla\omega_n^\nu(t)\|_{L^2}^2dt \\
&\gtrsim 
\varphi(\nu_n)\frac{1}{T}\int_0^T\|\nabla\omega_n(t)\|_{L^2}^2dt
-\varphi(\nu_n) \sup_{0<t<T}\|\nabla(\omega_n^\nu(t)-\omega_n(t))\|_{L^2}^2\\
&
\gtrsim_T
\varphi(\nu_n)S_n^{2c_0}-\varphi(\nu_n).
\end{split}
\end{equation*}
Here we set $\varphi(\nu_n)\sim(-\log(\nu_n))^{-c_0}$,
then we have the desired estimate.
\end{proof}

\section{Creation of Lagrangian deformation}\label{LD near the origin}

In this section, we give the proofs of two propositions. 

\begin{proof}[Proof of Proposition \ref{prop:LLD-bubbles}]

Recall that the 2D flow map associated with the initial vorticity $\omega_{0,n}$ is denoted by $\eta_n=(\eta_{n,1}, \eta_{n,2})$ and it satisfies
\begin{equation*}
\begin{split}
\partial_t \eta_n(t,x)=u_n(t,\eta_n(t,x))\quad\text{with}\quad \eta_n(0,x)=x\in\mathbb{T}^2,
\end{split}
\end{equation*} 
where $u_n$ is the corresponding velocity associated with $\omega_{0,n}$.
Then we have the following representation 
\begin{equation}\label{eq:rep}
D\eta_n:=
\begin{pmatrix}
\partial_1\eta_{n,1}&\partial_2\eta_{n,1}\\
\partial_1\eta_{n,2}&\partial_2\eta_{n,2}\\
\end{pmatrix}, \quad 
D\eta_n^{-1}=
\begin{pmatrix}
\partial_2\eta_{n,2}&-\partial_2\eta_{n,1}\\
-\partial_1\eta_{n,2}&\partial_1\eta_{n,1}\\
\end{pmatrix}.
\end{equation}
 Note that we always have $\nrm{D\eta_n(t)}_{L^\infty} = \nrm{D\eta_n^{-1}(t)}_{L^\infty}$.
	To begin with, we shall fix some $n$ large and write $\eta = \eta_n$ for simplicity. Moreover, we note that  $S_n$ is divergent in $n$.
We systematically use the following ``Key Lemma'' due to Kiselev and Sverak \cite{KS} (a version on $\mathbb{T}^2$ has been established by \cite{Z}). The version presented below written in the polar coordinates is given in \cite[Lemma 5.1]{EJSVP}.  
 
 \medskip 
 
\noindent \textbf{Key Lemma}. \textit{ Assume the vorticity on $\mathbb{T}^2$ is bounded and odd with respect to both axis. Then $u = \nabla^\perp\Delta^{-1}\omega$ satisfies \begin{equation}\label{eq:keyLemma}
	\begin{split}
	u(t,r,\theta) = \begin{pmatrix}
	\cos\theta \\
	-\sin\theta
	\end{pmatrix} r I(t,r) + r B(t,r,\theta)
	\end{split}
	\end{equation} for $|r| \le 1/2$, where \begin{equation*}
	\begin{split}
	I(t,r) := \frac{4}{\pi}\int_0^{\pi/2}\int_{  2r}^{1 } \frac{\sin(2\theta')}{s} \omega(t,s,\theta') dsd\theta'
	\end{split}
	\end{equation*} and \begin{equation}\label{absolute constant}
	\begin{split}
	\nrm{B(t)}_{L^\infty} \le C \nrm{\omega(t)}_{L^\infty}
	\end{split}
	\end{equation} for some absolute constant $C > 0$. }

	\medskip

	\noindent For simplicity, we shall set $I(t) := I(t,0)$ which is well-defined as the vorticities we consider is always vanishing in a small neighborhood of the origin. We now give a brief outline of the argument. The goal is to estimate the time integration of the ``key integral'' $I(t)$  {for some time interval $[0,t^*]$,} since then we deduce from the ODE \begin{equation*}
	\begin{split}
	\frac{d}{dt} \rd_2\eta_{n,2}(t,0) = \rd_2 u_{n,2}(t,0) \rd_2\eta_{n,2}(t,0),\quad \rd_2\eta_{n,2}(0,0)=1
	\end{split}
	\end{equation*} and \begin{equation*}
	\begin{split}
	\rd_2 u_{n,2}(t,0) \ge I(t,0)-C
	\end{split}
	\end{equation*} that (assuming $I(t,0)>2C$) \begin{equation*}
	\begin{split}
	|\rd_1\eta_{n,1}(t^*,0)| \ge\exp\left(\frac{1}{2} \int_0^{t^*} I(t) dt \right) .
	\end{split}
	\end{equation*}  {The assumption 
\begin{equation}\label{C}
I(t,0)>2C
\end{equation}
is easily guaranteed by taking $\ell_0$ large, since the contribution to $I$ from $\tilde{\omg}_0$ diverges as $\ell_{{0}}\rightarrow+\infty$ (see \cite{KS}).}
In turn, we may write \begin{equation*}
	\begin{split}
	I(t) = \sum_{ k = 1}^n I_k(t)+I_{2n}(t)
	\end{split}
	\end{equation*} where
\begin{equation*}
	\begin{split}
	I_{2n} (t) := \frac{4}{\pi} \int_0^{\pi/2} \int_0^{1/2} \frac{\sin(2\theta)}{r}  \phi_{ {2n}}(\eta^{-1}(t))(r,\theta) drd\theta 
	\end{split}
	\end{equation*} 
and
 \begin{equation*}
	\begin{split}
	I_k (t) := \frac{4}{\pi} \int_0^{\pi/2} \int_0^{1/2} \frac{\sin(2\theta)}{r} a_k \phi_k(\eta^{-1}(t))(r,\theta) drd\theta 
	\end{split}
	\end{equation*} are the contributions to $I(t)$ from the bubble initially located at the scales $2^{-2n}$ and $2^{-k}$ respectively. 
	Our strategy is to establish the following assertion  inductively in $k$: The  ``shape'' of the $k$-th bubble essentially remains the same within the time scale $t_k := c_1/S_k$ for some absolute constant $c_1 > 0$. Here what we mean by shape will be made precise {later};  {for now we just mention that as a consequence, it follows that  $I_k(t)\ge cI_k(0)$} {for $t\in [0,t_k]$,}  {where $c>0$ is a universal constant.} Assuming for a moment that this statement holds,  we obtain that \begin{equation*}
	\begin{split}
	\int_0^{t_k} I_k(t) dt \gtrsim t_k I_k(0) \gtrsim \frac{a_k}{S_k}.
	\end{split}
	\end{equation*}  Then we have
(since $I_{2n}(t)$ is non-negative, we just drop it off)
  \begin{equation*}
	\begin{split}
	\int_0^{t^*} I(t) dt \ge \sum_{k=1}^{n}\int_0^{\min\{t_k,t^*\}} I_k(t) dt \gtrsim \sum_{1\le k \le n, c_1 < t^*S_k} \frac{a_k}{S_k} 
	\end{split}
	\end{equation*} owing to the non-negativity of each $I_k(t)$. We now observe that, by approximating the sum with a Riemann integral, \begin{equation*}
	\begin{split}
	\sum_{k=1}^n \frac{a_k}{S_k} \approx \log(S_n) 
	\end{split}
	\end{equation*} and taking $k^*$ be the smallest number satisfying \begin{equation*}
	\begin{split}
	S_{k^*}t^* > c_1,
	\end{split}
	\end{equation*} (which exists by taking $n$ larger if necessary, since $t^* > 0$ and we are assuming that the sequence $S_{k^*}$ is divergent for $k^*\to\infty$) \begin{equation*}
	\begin{split}
	\sum_{k=1}^{k^*} \frac{a_k}{S_k} \approx \log(S_{k^*}) \approx \log(\frac{c_1}{t^*}). 
	\end{split}
	\end{equation*} This gives \eqref{eq:LLD-bubbles} with some positive constant $c_0$. Similarly, we obtain that as long as $t\gtrsim S_n$, \begin{equation*}
		\begin{split}
		\rd_2 u_{n,2}(t,0) \gtrsim t^{-1},
		\end{split}
		\end{equation*} which is simply \eqref{eq:velgrad-large}. Hence it is sufficient to prove that \begin{equation*}
	\begin{split}
	I_k(t) \gtrsim I_k(0),\qquad t \in [0,t_k]
	\end{split}
	\end{equation*} uniformly in $k$. Below we shall formulate and prove a claim which implies the above lower bound.

	\medskip
	
	\noindent \textit{Step I: some preparations}
	
	\medskip

	\noindent We make some simple observations regarding the evolution of the bubbles. Recall from the definition of $\phi_0$ that restricted on to the positive quadrant, there exist ``rectangles'' $$\overline{R}_0 = \left\{ (r,\theta) : \overline{r}_1 < r < \overline{r}_2, \overline{\theta}_1 < \theta < \overline{\theta}_2 \right\}$$ and $$\underline{R}_0= \left\{ (r,\theta) : \underline{r}_1 < r < \underline{r}_2, \underline{\theta}_1 < \theta < \underline{\theta}_2 \right\}$$ such that \begin{equation*}
	\begin{split}
	\phi_0 = 1 \mbox{ on }  \underline{R}_0 \quad \mbox{ and } \quad \phi_0 = 0  \mbox{ outside of }  \overline{R}_0. 
	\end{split}
	\end{equation*} We may set \begin{equation*}
	\begin{split}
	\frac{1}{2} < \overline{r}_1 < \underline{r}_1 < \underline{r}_2 < \overline{r}_2 < 2 
	\end{split}
	\end{equation*} and \begin{equation*}
	\begin{split}
	\frac{\pi}{6} < \overline{\theta}_1 < \underline{\theta}_1 < \underline{\theta}_2 < \overline{\theta}_2 < \frac{\pi}{3} . 
	\end{split}
	\end{equation*} Now by simple scaling, with the $2^{-k}$-scaled rectangles $\underline{R}_k$ and $\overline{R}_k$, we have \begin{equation*}
	\begin{split}
	\phi_k = 1 \mbox{ on }  \underline{R}_k \quad \mbox{ and } \quad \phi_k = 0  \mbox{ outside of }  \overline{R}_k, 
	\end{split}
	\end{equation*} still restricted on the first quadrant (more precisely on $[0,1]^2$). This time, take an even smaller rectangle: \begin{equation*}
	\begin{split}
	\underline{R}_0^* = \left\{ (r,\theta) : \underline{r}_1^* < r < \underline{r}_2^*, \underline{\theta}_1^* < \theta < \underline{\theta}_2^* \right\} \subset \underline{R}_0
	\end{split}
	\end{equation*} where we may set \begin{equation*}
	\begin{split}
	\underline{r}_1^* = \frac{2\underline{r}_1 + \underline{r}_2}{3}, \quad \underline{r}_2^* = \frac{ \underline{r}_1 + 2\underline{r}_2}{3}
	\end{split}
	\end{equation*} and similarly \begin{equation*}
	\begin{split}
	\underline{\theta}_1^* = \frac{2\underline{\theta}_1 + \underline{\theta}_2}{3}, \quad \underline{\theta}_2^* = \frac{ \underline{\theta}_1 + 2\underline{\theta}_2}{3}. 
	\end{split}
	\end{equation*} Then as before define \begin{equation*}
	\begin{split}
	\underline{R}_k^* := \left\{ (r,\theta) : \underline{r}_1^* < 2^kr < \underline{r}_2^*, \underline{\theta}_1^* < \theta < \underline{\theta}_2^* \right\}. 
	\end{split}
	\end{equation*} Moreover, define \begin{equation*}
	\begin{split}
	\overline{A}_k := \{ (r,\theta) : 2^{-k-1} < r < 2^{1-k} , 0 < \theta < \frac{\pi}{2}  \}. 
	\end{split}
	\end{equation*} We shall now prove the following
	
	\smallskip
	
	\noindent \textbf{Claim.}  In the time interval $[0,t_k]$, the $k$-th bubble remains $a_k$ on the rectangle $\underline{R}_k^*$ and vanishes outside $\overline{A}_k$. Here $t_k := c_1/S_k$ with $c_1 > 0$ independent of $k$.
	
	\smallskip	
	
	\noindent This is what we mean by retaining the same ``shape''.
	We now rewrite the evolution of the trajectories in polar coordinates, using \eqref{eq:keyLemma}. Given some $x \in [0,1]^2$, we shall express the point $\eta(t,x)$ using $|\eta|$ and $\theta(\eta)$. Then, \begin{equation}\label{eq:traj-polar-rad}
	\begin{split}
	\frac{d}{dt}|\eta| & = u(t,\eta) \cdot \begin{pmatrix}
	\cos(\theta(\eta)) \\
	\sin(\theta(\eta)) 
	\end{pmatrix} \\
	& = |\eta| \left( \cos(2\theta(\eta)) I(t,|\eta|) + (\cos(\theta( \eta))B_1 + \sin(\theta(\eta))B_2) \right) 
	\end{split}
	\end{equation} and \begin{equation}\label{eq:traj-polar-angle}
	\begin{split}
	\frac{d}{dt}\theta(\eta) & = \frac{u(t,\eta)}{|\eta|} \cdot \begin{pmatrix}
	-\sin(\theta(\eta)) \\
	\cos(\theta(\eta)) 
	\end{pmatrix} \\
	& = -\sin(2\theta(\eta))  I(t,|\eta|) +  (-\sin(\theta( \eta))B_1 + \cos(\theta(\eta))B_2)
	\end{split}
	\end{equation} where $B = (B_1, B_2)$ is from \eqref{eq:keyLemma}.

	\medskip
	
	\noindent \textit{Step II: induction base case $k = 1$}
	
	\medskip

	\noindent To proceed, we recall a simple estimate of Yudovich (see e.g. \cite{EJ} for a proof): \begin{lemma}\label{lem:Yud}
	Let $\omega(t) \in L^\infty([0,\infty): L^\infty(\mathbb{T}^2))$ be a solution of the 2D Euler equations on $\mathbb{T}^2$, and $\eta$ be the associated flow map. Then for some absolute constant $c > 0$, we have \begin{equation}\label{eq:Yud}
	\begin{split}
	|x-x'|^{1+ct\nrm{\omega_0}_{  L^\infty}} \le |\eta(t,x) - \eta(t,x')| \le |x-x'|^{1-ct\nrm{\omega_0}_{  L^\infty}},
	\end{split}
	\end{equation} for all $0 \le t \le 1$ and $|x-x'| \le 1/2$.
\end{lemma}

	We shall be concerned with the bubble $\phi_1$ and the trajectories $\eta(t,x)$ where $x \in \supp(\phi_1)$. Using the Yudovich estimate \eqref{eq:Yud} with $x' = 0$, we see that such trajectories are trapped inside the region $\{ 2^{-2} \le r \}$ during $[0,t_1]$ by choosing $c_1 > 0$ depending only on  $c\nrm{\omega_0}_{L^\infty}$ in \eqref{eq:Yud}. Similarly, trajectories starting from $\cup_{k>1}\supp(\phi_k)$ cannot cross the circle $\{ r = 2^{-2}\}$. This results in a naive bound \begin{equation*}
	\begin{split}
	I_1(t,|\eta|) \le I_1(t,2^{-2}) \lesssim 
a_12^{-2}\|\phi_1(\eta^{-1}(\cdot))\|_{L^1}\lesssim
a_1 
	\end{split}
	\end{equation*} on the same time interval, 
where
\begin{equation}\label{key-integral}
	\begin{split}
	I_k (t,r) := \frac{4}{\pi} \int_0^{\pi/2} \int_{2r}^{1/2} \frac{\sin(2\theta)}{s} a_k \phi_k(\eta^{-1}(t))(s,\theta) dsd\theta. 
	\end{split}
\end{equation}
 We use this to obtain slightly improved estimates on $|\eta|$ in \eqref{eq:traj-polar-rad}: \begin{equation*}
	\begin{split}
	\left| \frac{d}{dt} \ln \frac{1}{|\eta|} \right| \lesssim S_1  
	\end{split}
	\end{equation*} using $|B| \lesssim 1$. This guarantees that, given any small $\epsilon > 0$, by taking $c_1 = c_1(\epsilon) > 0$  small enough if necessary, we have \begin{equation*} 
	\begin{split}
	\left|\ln \frac{|\eta(0,x)|}{|\eta(t,x)|}\right| < \epsilon, \quad t\in [0,t_1],
	\end{split}
	\end{equation*}  recalling that $t_1 = c_1/S_1$. For the angle, we simply use \begin{equation*}
	\begin{split}
	\left| \frac{d}{dt} \theta(\eta)  \right| \lesssim S_1 
	\end{split}
	\end{equation*} to deduce \begin{equation*}
	\begin{split}
	\left| \theta(\eta(t,x))  - \theta(\eta(0,x)) \right| < \epsilon
	\end{split}
	\end{equation*} again for $t\in [0,t_1]$ by taking $c_1 > 0$ smaller if necessary. Thus, a suitable choice of $\epsilon > 0 $ (depending only on $\underline{r}_{1}, \underline{r}_{2}, \underline{\theta}_1, \underline{\theta}_2$) finishes the proof of the \textbf{Claim} for the case $k = 1$.
		
	\medskip
	
	\noindent \textit{Step III: completing the induction}
	
	\medskip
	
	\noindent We now assume that for some $k_0 > 1$ the \textbf{Claim} has been proved for $k = 1, \cdots, k_0-1$. We are now concerned with the trajectories $\eta(t,x)$ where $t \le t_{k_0}$ and $x \in \supp(\phi_{k_0})$. The induction hypothesis guarantees that, as long as $2^{-(k_0 + 1)} < |\eta| < 2^{-(k_0 -1)}$,
we have that \begin{equation*}
	\begin{split}
	\left| \frac{d}{dt} \ln \frac{1}{|\eta|} \right| \lesssim S_{k_0} 
	\end{split}
	\end{equation*}  {simply because $t_k$ is decreasing with $k$ and the hypothesis ensures that the contribution of $a_k\phi_k \circ \eta^{-1}(t)$ to the key integral 
\eqref{key-integral}
is bounded by $c a_k$ with some $c$ independent of $k$, for $k = 1, \cdots, k_0 - 1$. 
More precisely
\begin{equation*}
I_k(t,|\eta|)\lesssim a_k2^{-2k}\|\phi_k(\eta^{-1}(\cdot))\|_{L^1}\lesssim a_k.
\end{equation*}
Strictly saying, here, we use the even smaller rectangles
$\underline{R}^*_k$. Thus $c$ is depending on $\epsilon>0$. 
This implies that 
 \begin{equation*} 
	\begin{split}
	\left|\ln \frac{|\eta(0,x)|}{|\eta(t,x)|}\right| < \epsilon , \quad t\in [0,t_{k_0}]
	\end{split}
	\end{equation*} for the \textit{same} $\epsilon$ and $c_1$.} Similarly, we can deduce \begin{equation*}
	\begin{split}
	\left| \theta(\eta(t,x))  - \theta(\eta(0,x)) \right| < \epsilon
	\end{split}
	\end{equation*} on $[0,t_{k_0}]$. The proof of \textbf{Claim} is complete, which finishes the proof. 
\end{proof}
  
\begin{remark} A few remarks are in order.
	\begin{itemize}
		\item \textit{Large Lagrangian deformation occurs at the origin}. Proposition \ref{prop:LLD-bubbles} shows that for bubbles satisfying $S_n \rightarrow \infty$ as $n \rightarrow \infty$, large Lagrangian deformation must occur, and it occurs even within a time interval that shrinks to zero for $n$ large. We emphasize that we can pinpoint the location of large Lagrangian deformation to be the origin (which was an open problem to the best of our knowledge), while using contradiction arguments it is possible (see \cite{BL,EJ}), with less work, to show existence of large Lagrangian deformation (somewhere in the domain). 

		\item \textit{Dichotomy for bubbles}. Note that in the case when the sequence $a_k$ is summable, the initial vorticity belongs to the critical Besov space $B^0_{\infty,1}$ uniformly in $n$ (for the rigorous calculation, see \cite{MY} for example). There is uniqueness and existence in this space $B^0_{\infty,1}$ (\cite{V}), which in particular guarantees that the corresponding velocity gradient is uniformly bounded in $n$ for a short time interval. Therefore, we have the following dichotomy for bubbles: short-time large Lagrangian deformation occurs \textit{if and only if} the sequence $\{ a_k \}$ is not summable. 
		

		\item \textit{Sharpness of the growth rate}. It can be shown that with the data in \eqref{eq:bubbles}, we have \begin{equation*}
		\begin{split}
		|D\eta(t,0)| \le C(c_2),\qquad t \in [0,c_2/S_n]
		\end{split}
		\end{equation*} for any fixed constant $c_2 > 0$. This follows from the well-posedness of the Euler equations with vorticity in $B^0_{\infty,1}$ (\cite{PP}) and the fact that $\nrm{\omega_{0,n}}_{B^0_{\infty,1}} \sim S_n$. Comparing this with \eqref{eq:LLD-bubbles}, one sees that the lower bound is sharp at least during this time scale. Hence we must wait a bit longer to see large deformation at the origin. 
		

		\item \textit{Case of the continuum}. Our considerations equally apply well to the ``continuum'' version of the bubbles; that is, we may take locally \begin{equation*}
		\begin{split}
		\omega_{0,n}(r,\theta) = \varphi_{2^{-n-1}} * (g(r) \chi(\theta) ), \quad 0 \le r < 1/2
		\end{split}
		\end{equation*} where $\chi \ge 0$ on $\theta \in [0,\pi/2]$ and $\chi(\theta) = -\chi(-\theta) = -\chi(\pi-\theta)$, and $g \ge 0$ is a bounded continuous function on $[0,1/2] \rightarrow [0,1]$. Here $a_k$ corresponds to $g(2^{-k})$ and $S_k$ to $\int_{2^{-k}}^{1} g(r)r^{-1}dr$. For an example, in the case $g(r) = |\ln r|^{-1/2-\eps}$, $\omega_0 = g(r)\sin(2\theta)$ belongs to $H^1$ (considered explicitly in \cite{EJ}), and using the method in this paper one can show that the corresponding solution escapes $H^1$ without appealing to a contradiction argument. 
		
	\end{itemize}

\end{remark}

\begin{proof}[Proof of Proposition \ref{prop:LLD2}]
	Recall that Proposition \ref{prop:LLD-bubbles} established creation of large Lagrangian deformation at the origin for the solution of 2D Euler with initial data $\omega_{0,n}$. We now prove that the deformation persists on the support of $\phi_{2n}$ and use this observation to complete the proof of Proposition \ref{prop:LLD2}. 
	
	For convenience, we set $\omega_{0,n}^S = \phi_{2n}$ and $\omega_n^S(t) := \omega_{0,n}^S \circ \eta_n^{-1}(t)$. We observe that for any $t$, $\nrm{\omega_n(t)}_{H^1}\ge \nrm{\omega_n^S(t)}_{H^1}$. From the definition, $\nrm{\omega_{0,n}^S}_{H^1} \gtrsim 1$. Moreover, we note that for $t\in [0,\delta]$, the support of $\omega_n^S(t)$ is contained in the ball of radius $2^{-n}$ centered at the origin. This can be easily shown using Lemma \ref{lem:Yud} (we choose $\delta=\frac{1}{2c\|\omega_0\|_{L^\infty}}$). 
	Recalling the simple estimate \begin{equation*}
	\begin{split}
	\nrm{\nb u_n(t)}_{C^{\frac{1}{2}}} \lesssim \nrm{\omega_n(t)}_{C^{\frac{1}{2}}} \lesssim 2^{\frac{n}{2}}, 
	\end{split}
	\end{equation*} we have \begin{equation*}
	\begin{split}
	|\nb u_n(t,x) - \nb u_n(t,0)| \lesssim 2^{\frac{n}{2}} |x|^{\frac{1}{2}},
	\end{split}
	\end{equation*} so that for $x \in \mathrm{supp}(\omega_n^S(t))$, \begin{equation*}
	\begin{split}
	|\nb u_n(t,x) - \nb u_n(t,0)| \lesssim 1. 
	\end{split}
	\end{equation*} In particular, since $\rd_1 u_{n,2}(t,0) = \rd_{2}u_{n,1}(t,0) = 0$ by the odd symmetry, \begin{equation*}
	\begin{split}
	|\rd_1 u_{n,2}(t,x)| + |\rd_2 u_{n,1}(t,x)| \lesssim 1. 
	\end{split}
	\end{equation*}
	
	We now fix some $x \in \mathrm{supp}(\omega_{0,n}^S)$, and estimate (dropping the subscript $n$ for simplicity) \begin{equation*}
	\begin{split}
	\frac{d}{dt}  \nb \eta(t,x)   = \nb u(t,0) \nb \eta(t,x)  + (\nb u(t,\eta) - \nb u(t,0))\nb\eta(t,x). 
	\end{split}
	\end{equation*} This can be estimated as follows: \begin{equation*}
	\begin{split}
	\frac{d}{dt}  |\nb \eta(t,x)| \ge (|\nb u(t,0)| - C) |\nb \eta(t,x)|.
	\end{split}
	\end{equation*} Since $\nb \eta(t,x)  = \mathrm{Id}$ at $t = 0$, for sufficiently large $n$, one can show easily from the lower bound of $\rd_1u_1(t,0)$ that  \begin{equation*}
	\begin{split}
	|\rd_1\eta_{n,1}(t,x)| = |\rd_2\eta_{n,2}(t,x)| \gtrsim (S_nt)^{\frac{c_0}{2}}. 
	\end{split}
	\end{equation*}
	
	Finally, \begin{equation*}
	\begin{split}
	\nb \omega_n^S(t) = \nb\omega_{0,n}^S \circ \eta_n^{-1} \nb\eta_n^{-1}
	\end{split}
	\end{equation*} and recalling the representation formula for $\nb\eta_n^{-1}$ in terms of $\nb\eta_n$ (see \eqref{eq:rep}) and $\nrm{\nb\omega_{0,n}^S}_{L^2}\gtrsim 1$, we obtain \begin{equation*}
	\begin{split}
	\nrm{\omega_n(t)}_{H^1}\ge \nrm{\nb \omega_n^S(t)}_{L^2} \gtrsim \inf_{x \in \mathrm{supp}(\omega_{0,n}^S)}|\rd_1\eta_{n,1}|\gtrsim (S_nt)^{\frac{c_0}{2}},
	\end{split}
	\end{equation*} which finishes the proof. \end{proof}

\section{Conclusion}

We prepared small-scale vortex blob and multi-scale odd-symmetric
vortices for the initial data, and showed that the corresponding
2D Euler flow creates vortex thinning. In turn, using this thinning, we showed that the corresponding 2D Navier-Stokes flow induces the enstrophy dissipation with strictly slower decaying rate than $Re^{-1}$. 


\vspace{0.5cm}
\noindent
{\bf Acknowledgments.}\ 
We thank Professor A. Mazzucato for inspiring communications and telling us about the article \cite{FML}. 
Research of TY  was partially supported by 
Grant-in-Aid for Young Scientists A (17H04825),
Grant-in-Aid for Scientific Research B (15H03621, 17H02860, 18H01136, 18H01135 and  20H01819), 
Japan Society for the Promotion of Science (JSPS). IJ has been supported by the POSCO Science Fellowship of POSCO TJ Park Foundation  and the National Research Foundation of Korea (NRF) grant (No. 2019R1F1A1058486).  
\bibliographystyle{amsplain} 

\end{document}